\pdfoutput=1

\documentclass[11pt]{amsart}

\usepackage[bookmarks=true,%
    colorlinks=true,%
    linkcolor=blue,%
    citecolor=blue,%
    filecolor=blue,%
    menucolor=blue,%
    urlcolor=blue,%
    breaklinks=true]{hyperref}

\usepackage{bm}
\usepackage{amsthm, cleveref, amssymb, enumerate}
\usepackage{graphicx}

\newcommand{\Or}{\mathrm{O}}
\newcommand{\SO}{\mathrm{SO}}
\newcommand{\Vol}{\mathrm{Vol}}

\renewcommand{\H}{\mathbb{H}}
\newcommand{\Q}{\mathbb{Q}}
\newcommand{\Z}{\mathbb{Z}}
\newcommand{\R}{\mathbb{R}}

\newcommand{\Tor}{\mathrm{Tor}}

\newcommand{\GL}{\mathrm{GL}}

\newtheorem{theorem}{Theorem}[section]
\newtheorem{lemma}[theorem]{Lemma}
\newtheorem{corollary}[theorem]{Corollary}
\newtheorem{prop}[theorem]{Proposition}
\newtheorem{conj}[theorem]{Conjecture}
\newtheorem{rem}[theorem]{Remark}
\newtheorem{quest}[theorem]{Question}

\crefname{theorem}{Theorem}{Theorems}
\crefname{lemma}{Lemma}{Lemmas}
\crefname{corollary}{Corollary}{Corrollaries}
\crefname{prop}{Proposition}{Propositions}
\crefname{conj}{Conjecture}{Conjectures}
\crefname{rem}{Remark}{Remarks}
\crefname{quest}{Question}{Questions}
\crefname{section}{Section}{Sections}
\crefname{equation}{Equation}{Equations}
\crefname{figure}{Figure}{Figures}

\newcommand{\legendre}[2]{\genfrac{(}{)}{}{}{#1}{#2}}

\usepackage{color}

\title[Embedding closed 3-manifolds in small volume hyperbolic 4-manifolds]{Embedding closed hyperbolic 3-manifolds in small volume hyperbolic 4-manifolds}
\author{M. Chu}
\author{A. W. Reid}

\address{\newline
Department of Mathematics,\newline
University of Illinois at Chicago,\newline
Chicago, IL 60607, USA.}
\email{michu@uic.edu }
\address{\newline
Department of Mathematics,\newline
Rice University,\newline
Houston, TX 77005, USA}
\email{alan.reid@rice.edu}

\thanks{First author supported by NSF grant DMS $1803094$ and the second author  by NSF grant DMS $1812397$}

\begin{document}

\begin{abstract}
In this paper we study existence and lack thereof of closed embedded orientable co-dimension one totally geodesic submanifolds of minimal volume cusped orientable hyperbolic manifolds.
\end{abstract}

%\subjclass{20E26, 20E18 (20F65, 20F10, 57M25) }


\maketitle

%%%%%%%%%%%%%%%%%%%%%%%%%%%%%%%%%%%%%%%%%%%%%%%%%%%%%
%
%
% introduction
%
%
%%%%%%%%%%%%%%%%%%%%%%%%%%%%%%%%%%%%%%%%%%%%%%%%%%%%%

\section{Introduction}
\label{intro}
Let $W^n$ denote a minimal volume, orientable, cusped hyperbolic $n$-manifold. In this paper we will be concerned with the existence of closed embedded orientable totally geodesic hyperbolic submanifolds $M^{n-1} \hookrightarrow W^n$.

When $n=2$, $W^2$ is either the thrice-punctured sphere or a once-punctured torus, and in the former case there are no such submanifolds, whilst in the latter there are many. In dimension $3$, by work of Cao and Meyerhoff,
\cite{CM}, the manifolds in question are the complement of the figure-eight knot and its sister. These do not contain any closed embedded orientable totally geodesic surfaces (although they do contain infinitely many immersed closed totally geodesic surfaces \cite[Chapter 9]{MR}). Indeed, they do not contain any closed embedded
essential surfaces (see \cite{Th}, \cite{CJR} and \cite{FH}).  These 3-manifolds are arithmetic hyperbolic 3-manifolds, as is the case of the thrice-punctured sphere. The once-punctured torus has a positive dimensional Teichm\"uller space,
but there are a finite number of examples which are arithmetic.  

This paper is concerned with the following conjecture.  As we discuss in \cref{reduce}, \cref{noclosedembedded} is easily reduced to the only non-trivial case, that of dimension $4$.

\begin{conj}
\label{noclosedembedded}
Let $W^n$ denote a minimal volume, orientable, cusped arithmetic hyperbolic $n$-manifold. If $W^n$ contains a co-dimension one, closed embedded orientable totally geodesic submanifold, then $n=2$ and $W^2$ is an arithmetic once-punctured torus.\end{conj}

In \cite{RT} the authors provide a census of $1171$ so-called integral congruence two hyperbolic $4$-manifolds that are all obtained from face-pairings of the ideal $24$-cell in $\mathbb{H}^4$. These are all commensurable cusped, arithmetic, hyperbolic 4-manifolds of Euler characteristic $1$ (i.e. minimal volume). Amongst
these, only $22$ are orientable, and these are listed in \cref{RTmanifolds} together with some information that we will make use of.  Towards a positive resolution of \cref{noclosedembedded} we prove the following result.
\begin{theorem}
\label{main_or}
Let $W$ denote one of the $22$ manifolds mentioned above. Then $W$ does not contain a closed embedded orientable totally geodesic hyperbolic $3$-manifold.\end{theorem}
As remarked upon, the $1171$ integral congruence two hyperbolic $4$-manifolds are all commensurable. In a private communication J. Ratcliffe and S. Tschantz have informed us that 
there are many more manifolds obtained by side-pairings of the ideal $24$-cell in this commensurability class, namely $13,108$ side-pairings of the ideal $24$-cell (up to symmetry of the $24$-cell) yield a cusped hyperbolic 4-manifold of Euler characteristic $1$. Only 
$675$ of these side-pairings provide orientable manifolds (which include the $22$ orientable ones of Theorem \ref{main_or}).  
In addition, in \cite{RiSl}
it is shown that there is at least one more commensurability class of cusped arithmetic hyperbolic $4$-manifolds that contains an orientable cusped hyperbolic $4$-manifold with Euler characteristic $1$. 

At present we cannot say anything about \cref{noclosedembedded} for these other orientable examples, nor
do we have a classification of the finite number of commensurability classes of cusped arithmetic hyperbolic 4-manifolds that contain a manifold with Euler characteristic $1$ (although in principle this should be doable).

Our methods also apply to another situation.  In \cite{I} the author provides an example of a cusped, orientable, hyperbolic 4-manifold of Euler characteristic $2$ that is the complement of five $2$-tori in $S^4$ (with the standard smooth structure, \cite{I2}). This link complement
arises as the orientable double cover of the non-orientable manifold $1011$ in the census of integral congruence two hyperbolic $4$-manifolds mentioned above (see also \cite{I}).
We prove the following result.
\begin{theorem}
\label{main_1011}
Let $W$ be the link complement in $S^4$ described above. Then $W$ contains an embedded orientable totally geodesic cusped hyperbolic $3$-manifold isometric to the complement of the link $8^3_9$ (shown in \cref{fig: links}) in $S^3$, but no closed orientable embedded totally geodesic hyperbolic $3$-manifold.
\end{theorem}
A simple but elegant argument (see \cite[Proposition 4.10]{I0}) shows that if $X$ is a hyperbolic link complement of $2$-tori in $S^4$, then $\chi(X)=\chi(S^4)=2$, and so
there are {\bf only finitely many} hyperbolic link complements of $2$-tori in $S^4$. A similar statement holds more generally for link complements of $2$-tori and Klein bottles in other fixed $4$-manifolds.
In \cite{IRT}, four additional examples of link complements of $2$-tori were found in manifolds homeomorphic to $S^4$. These arise as the orientable double covers of the non-orientable manifolds
in the census of \cite{RT} with numbers $23$, $71$, $1091$ and $1092$.  In a forthcoming paper \cite{CR}, we will address the existence of closed embedded totally geodesic hyperbolic $3$-manifold in these examples. 
This requires additional techniques, 

By way of comparison, Thurston's hyperbolization theorem shows many links in $S^3$ have hyperbolic complements, and although it is known that many hyperbolic link complements in $S^3$ do not contain a closed embedded totally geodesic surface (e.g. alternating links \cite{MeR}), examples do exist (see \cite{Lei} and \cite{MeR}).  

Finally, we point out if one merely asks for a smooth embedding of a closed orientable $3$-manifold into $S^4$ then there are obstructions; for example it is a result of Hantzsche \cite{Han} that if a closed orientable 3-manifold $M$ embeds in $S^4$, then 
$\Tor(H_1(M,\Z))\cong A\oplus A$ for some finite abelian group $A$. In fact, the Kirby Problem List Question 3.20 \cite{Kir} asks: {\em Under what conditions does a closed, orientable 3-manifold $M$ smoothly embed in $S^4$?}
We refer the reader to \cite{BB} for examples, for more discussion of this question, and additional references.

\section*{Acknowledgements}
{\em The authors would like to thank the Institut de Math\'ematiques, Universit\'e de Neuch\^atel for its hospitality during the early stages of this work. They benefited greatly from many helpful conversations with  A. Kolpakov and S. Riolo.  They would also like to thank D.D. Long and L. Slavich for helpful conversations on topics related to this work. We are most grateful to the referee for helpful comments, and in particular for asking a question that led to the discovery of an oversight in the first
version of this paper.}

\section{Cusped arithmetic hyperbolic manifolds}
\label{s:cusp}

We will mainly work with the hyperboloid model of ${\mathbb H}^n$ defined using the quadratic form $j_n$ defined as $x_0^2+x_1^2+\ldots x_{n-1}^2-x_n^2$; i.e.
$${\mathbb H}^n = \{x=(x_0,x_1,\ldots, x_n) \in {\mathbb R}^{n+1} : j_n(x)= -1, x_{n}>0\}$$
equipped with the Riemannian metric induced from the Lorentzian inner product associated to $j_n$. The full group of isometries of ${\mathbb H}^n$ is then identified with $\Or^+(n,1)$, the subgroup of
\[
\Or(n,1) = \{A \in \GL(n+1,\mathbb{R}) : A^tJ_nA = J_n\},
\]
preserving the upper sheet of the hyperboloid $j_n(x)=-1$, and
where $J_n$ is the symmetric matrix associated to the quadratic form $j_n$. The full group of orientation-preserving isometries is given by $\SO^+(n,1) = \{A\in \Or^+(n,1) : \det(A)= 1\}$.

\subsection{Constructing cusped arithmetic hyperbolic manifolds}\label{cusp}
Cusped, arithmetic, hyperbolic $n$-manifold are constructed as follows (see \cite{VS} for example).
Suppose that $X = \mathbb{H}^n/\Gamma$ is a finite volume cusped hyperbolic $n$-manifold. Then $X$ is arithmetic if $\Gamma$ is commensurable with a group $\Lambda < \SO^+(n,1)$ as described below.

Let $f$ be a non-degenerate quadratic form defined over $\Q$ of signature $(n,1)$, which we can assume is diagonal and has integer coefficients. Then $f$ is equivalent over $\mathbb{R}$ to the form $j_n$ defined above; i.e. there exists $T\in \GL(n+1,\mathbb{R})$ such that $T^tFT = J_n$, where $F$ and $J_n$ denote
the symmetric matrices associated to $f$ and $j_n$ respectively. Then $T^{-1}\SO(f,\Z)T\cap \SO^+(n,1)$ defines the arithmetic subgroup $\Lambda < \SO^+(n,1)$.

Note that the form $f$ is anisotropic (ie does not represent $0$ non-trivially over $\Q$) if and only if the group $\Gamma$ is cocompact, otherwise the group $\Gamma$ is non-cocompact (see \cite{BHCh}).  By Meyer's Theorem  
\cite[\S IV.3.2, Corollary 2]{Se}, the case that $f$ is anisotropic can only occur when $n=2,3$.

\subsection{Reducing \cref{noclosedembedded} to dimension 4}
\label{reduce}

We include a quick proof of the following, likely well-known, result.  Clearly \cref{highdim} then reduces \cref{noclosedembedded} to dimension $4$.

\begin{theorem}
\label{highdim}
Let $X=\mathbb{H}^n/\Gamma$ be a cusped, arithmetic, hyperbolic $n$-manifold with $n\geq 5$.   Then $X$ does not contain any co-dimension one, immersed, closed, totally geodesic hyperbolic manifold. On the other hand it contains infinitely many co-dimension one, immersed, cusped, totally geodesic hyperbolic manifolds.\end{theorem}

\begin{proof} 
The first part now follows easily from the discussion in \cref{cusp}, since any co-dimension one immersed totally geodesic hyperbolic manifold is arithmetic arising from a quadratic form of signature $(n-1,1)$. Since $n\geq 5$, such a form is isotropic
by Meyers theorem, and so the submanifold in question is non-compact.  

For the last claim we argue as follows. Assume that $f$ is a diagonal form defined over $\Q$ of signature $(n,1)$, $n\geq 5$. Then, we can restrict to a sub-quadratic form $f_1$ of signature $(n-1,1)$, which by hypothesis is isotropic. Using
\cref{cusp} we can use $f_1$ to build a non-cocompact subgroup $H<\Gamma$ with $H$ an arithmetic subgroup of $\Or^+(n-1,1)$. Now use the density of the commensurator of $\Gamma$ to construct infinitely many such groups $H$. \end{proof}

\section{Some background from \texorpdfstring{\cite{RT}}{RT}}\label{s:RTbackground}

\subsection{Integral congruence two hyperbolic 4-manifolds}\label{RTbackground}

For convenience we now set $J=j_4$. The manifolds of \cref{main_or} are all obtained by face-pairings of the regular ideal $24$-cell in ${\mathbb H}^4$ (with all dihedral angles $\pi/2$), and
arise as regular $({\Z}/2{\Z})^4$ covers of the orbifold ${\H}^4/\Gamma(2)$ where $\Gamma(2)$ is the level two congruence subgroup of the group $\Or^+(J,{\Z}) = \Or^+(4,1) \cap \Or(J,{\Z})$. These are the manifolds referred to
as {\em integral congruence two hyperbolic $4$-manifolds} in \cite{RT}.

It will be useful to describe the $({\Z}/2{\Z})^4$ action, and this is best described in the ball model as follows.
Locate the $24$-cell in the ball model of hyperbolic space with vertices 
$$(\pm1,0,0,0), (0,\pm 1,0,0), (0,0,0\pm 1,0), (0,0,0,\pm 1)~\hbox{and}~(\pm\frac{1}{2}, \pm\frac{1}{2}, \pm\frac{1}{2}, \pm\frac{1}{2}).$$ 
The four
reflections in the co-ordinate planes of ${\R}^4$ can be taken as generators of this $({\Z}/2{\Z})^4$ group of isometries. Passing to the hyperboloid model, these reflections are elements of $\Gamma(2)$ and are listed as the first four matrices in \cite[page 110]{RT}. Following \cite{RT} we denote
this $({\Z}/2{\Z})^4$ group of isometries by $K < \Gamma(2)$.
Note from \cite[Table 2]{RT}, only one of the $22$ examples under consideration admits a larger group of isometries (of order $48$) than that given by $K$. In particular,  none of these $22$ manifolds are regular covers
of the orbifold ${\H}^4/\Or^+(J,{\Z})$ (since $[\Or^+(J,{\Z}) : \Gamma(2)]=120$).

As noted in \cite{RT} (see also \cite{RT0}) all of the face-pairings of {\em any} of the integral congruence two hyperbolic $4$-manifolds are invariant under the group $K$. This implies that each of the coordinate hyperplane cross sections of the $24$-cell extends in each of the  integral congruence two hyperbolic $4$-manifolds to a totally geodesic hypersurface which is the fixed point set of one of the reflections described above. Following \cite{RT0} we call these hypersurfaces, {\em cross sections}.
As described in \cite{RT0}, these cross sections, can be identified with integral congruence two hyperbolic $3$-manifolds which are also described in \cite{RT}. Moreover, it is possible to use \cite{RT} to identify these explicitly in any given example.
The following can be deduced from \cite{RT} or \cite{RT0}

\begin{lemma}
\label{Xsection}
Any orientable cross-section is isometric to one of the complement in $S^3$ of the link $6^3_2$ (the Borromean rings), the link $8^3_9$, or the link $8^4_2$ (see \cref{fig: links}).\end{lemma}

\begin{proof} Each of the 3-dimensional cross sections must be isometric to one of the 107 possibilities encoded in \cite[Page 115]{RT}. However, these 107 are classified into 13 equivalence classes corresponding to isometry classes of the corresponding 3-manifolds \cite[Theorem 5]{RT}, out of which only 3 are orientable 3-manifolds. The 3 orientable possibilities are described in \cite[pages 108-109]{RT} and are the complement in $S^3$ of the link $6^3_2$, the link $8^3_9$ or the link $8^4_2$.  \end{proof}

\begin{figure}
 \includegraphics[width=.9\linewidth]{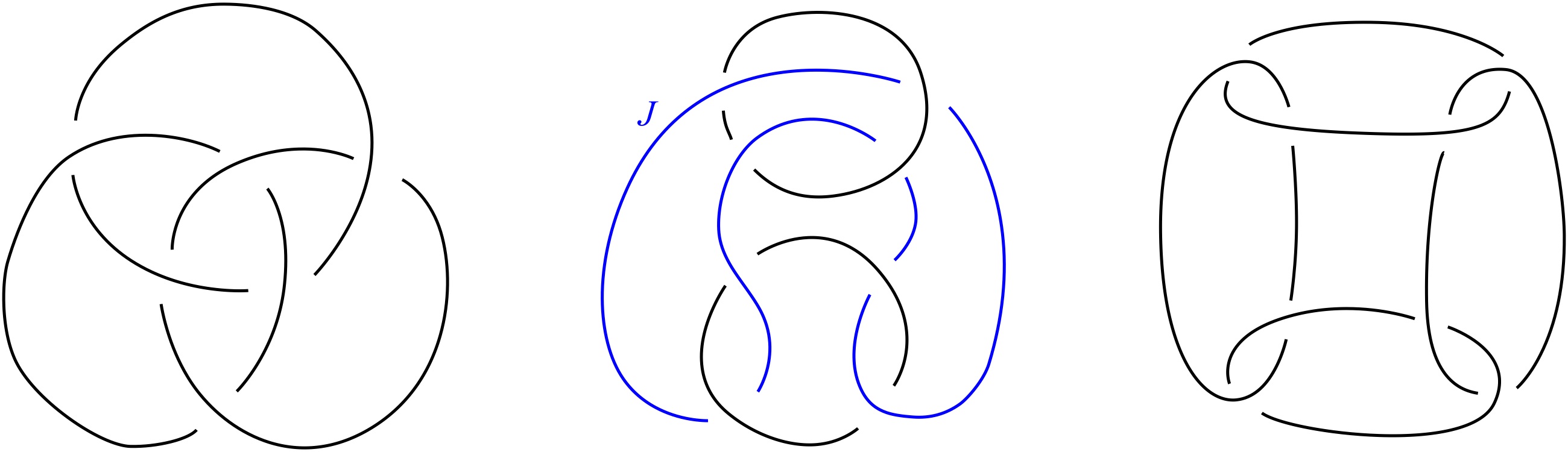}
 \caption{The links $6^3_2$, $8^3_9$, and $8^4_2$}
 \label{fig: links}
\end{figure}

\subsection{More about the links \texorpdfstring{$6^3_2$, $8^3_9$}{632, 839}, and \texorpdfstring{$8^4_2$}{842}}
Let $L$ denote one of the links $6^3_2$, $8^3_9$ or $8^4_2$.  Note that the complements of these links share the same $3$-dimensional hyperbolic volume which is approximately $7.3277247\ldots$.

\begin{lemma}
\label{small}
$S^3\setminus L$ does not contain a closed embedded totally geodesic 
surface.\end{lemma}

\begin{proof}  It is shown in \cite{Loz} that the complement of the Borromean rings is small (i.e. it does not contain any closed embedded essential surface), and in particular does not contain a closed embedded totally geodesic surface.

The link $8^4_2$ is the Montesinos link $K(\frac{1}{2},\frac{1}{2},\frac{1}{2},-\frac{1}{2})$ (in the notation of \cite{Oe}) and so \cite[Theorem 1]{Oe} implies that any closed embedded essential surface in the complement of $8^4_2$
arises from tubing the obvious $4$-punctured spheres separating pairs of tangles. In particular such a surface carries an accidental parabolic, and so cannot be totally geodesic. Indeed, in this case, it can be shown that in fact $8^4_2$ is small, as
these tubed surfaces compress.

Now consider the case of $L=8^3_9$, and suppose that $S^3\setminus L$ contains a closed embedded totally geodesic surface $S$. 
Trivial filling on the component $J$ in \cref{fig: links} provides a split link, and hence $S$ compresses in this filling.  In addition, $\pm 1$ fillings yields a manifold homeomorphic to the complement of the Whitehead link, which being a $2$-bridge link, does not contain any closed embedded essential surface \cite{HT}. Thus $S$ compresses in both $\pm 1$ filling.  Now the surface $S$ is totally geodesic (so does not carry an accidental parabolic element), and so an application of \cite[Theorem 1]{Wu} (following \cite{CGLS}) provides a contradiction since the slopes $\pm 1$ have distance $2$.\end{proof}

\section{Co-dimension one closed totally geodesic submanifolds in cusped arithmetic hyperbolic 4-manifolds}
\label{immersedin4}

In dimension $3$, any cusped arithmetic hyperbolic $3$-manifold contains infinitely many immersed closed totally geodesic surfaces (see \cite[Chapter 9]{MR}).  In this section, we show that the situation in dimension $4$ is
similar, providing a contrast with \cref{highdim} in dimensions $\geq 5$.

\subsection{Immersed closed totally geodesic hyperbolic 3-manifolds in integral congruence two hyperbolic 4-manifolds} 
We first show that the integral congruence two hyperbolic $4$-manifolds of \cite{RT} all contain many immersed closed totally geodesic hyperbolic $3$-manifolds (indeed any manifold in the commensurability class of these integral congruence two hyperbolic $4$-manifolds). To that end, let $p \equiv -1~\pmod 8$ be a prime, and $q_p$ the quadratic form (over $\Q$)
given by $x_1^2+x_2^2+x_3^2-px_4^2$. The congruence condition on $p$ implies that this form is anisotropic, and so as in \cref{cusp}, the group $\SO(q_p,\Z)$ determines a co-compact arithmetic lattice in $\SO^+(3,1)$.

\begin{prop}
\label{immersedRT}
With $p$ as above, and given any integral congruence two hyperbolic $4$-manifold $N$, there is a finite index subgroup $\Lambda_N< \SO(q_p,\Z)$ so that  $\mathbb{H}^3/\Lambda_N \hookrightarrow  \mathbb{H}^4/N$ is an immersed 
closed totally geodesic
hyperbolic 3-manifold. In particular any such $N$ contains infinitely many commensurability classes of immersed closed totally geodesic hyperbolic $3$-manifolds. \end{prop}

\begin{proof} The proof of the first claim will follow using standard arguments on equivalences of quadratic forms over $\Q$ yielding commensurable arithmetic lattices. In particular, we need to show that the quadratic form $q_p\perp <p>$ is equivalent to
the form $J$ of \cref{RTbackground} over $\Q$ (see \cite[\S 5,6]{ALR} for example). In this case, the equivalence can be seen directly as follows.

Let $p=8k-1$, and let
\[
A_p=\begin{pmatrix} 4k &4k-1\\ 4k-1 & 4k\end{pmatrix}, D = \begin{pmatrix} 1 & 0\\ 0 & -1\end{pmatrix}~\hbox{and}~D_p =  \begin{pmatrix} p & 0\\ 0 & -p\end{pmatrix} 
\]
A simple calculation shows that $A_pDA_p^t = D_p$, and from this the required equivalence can be deduced.

The second part follows from the fact that there are infinitely many primes $\equiv -1~\pmod 8$, and as noted above, all these quadratic forms being anisotropic over $\Q$ provide closed hyperbolic $3$-manifolds.\end{proof}

\subsection{Immersed closed totally geodesic hyperbolic $3$-manifolds in arithmetic hyperbolic $4$-manifolds} 
In this section we prove that the conclusion of \cref{immersedRT} holds much more broadly for cusped arithmetic hyperbolic $4$-manifolds.

\begin{theorem}
\label{allcusped}
Let $W$ be a cusped arithmetic hyperbolic $4$-manifold. Then $W$ contains infinitely many commensurability classes of immersed closed totally geodesic hyperbolic 3-manifolds.\end{theorem}

We begin with some preliminaries on non-degenerate diagonal quaternary quadratic forms 
\[
f_p=a_1x_1^2+a_2x_2^2+a_3x_3^2+a_4x_4^2
\]
over the local field $\Q_p$ for $p$ a prime or $p=-1$, with the understanding that $\Q_{-1}=\R$. 

Let $c_p(f)$ denote the Hasse-Minkowski invariant which is defined as
\begin{equation}
c_p(f) = \prod_{i<j}(a_i,a_j)_p
\end{equation}
where $(\cdot,\cdot)_p$ denotes the Hilbert symbol.  This invariant depends only on the equivalence class of $f$ and not on the choice of orthogonal basis. 

We collect some useful statements about Hilbert symbols and quadratic forms over local fields taken from \cite{Se}. Throughout $\legendre{u}{p}=(u,p)_p$ denotes the Legendre symbol, which as in \cite[II.3.3]{Se}, is extended to be defined for 
$u\in \Z_p^*$.

\begin{lemma}\label{lem Se} 
\begin{enumerate}[(a)]
\item If $p\neq 2$, the image of the integer $x=p^n u$ is a square in $\Q_p^*$ if and only if $n$ is even and $\legendre{u}{p}=1$ \cite[\S II.3.3, Theorem 3, page 17]{Se}.
\item The image of the integer $x=2^n u$ is a square in $\Q_2^*$ if and only if $n$ is even and $u\equiv 1\mod8$ \cite[II.3.3, Theorem 4, page 18]{Se}.
\item The Hilbert symbol satisfies the formulas \cite[III.1.1, Proposition 2, page 19]{Se}
	\begin{enumerate}[(i)]
		\item $(a,b)_p=(b,a)_p$
		\item$(a,b^2)_p=1$
		\item $(a,-a)_p=1$.
		\item $(-1,-1)_p=\begin{cases} -1 & \text{if } p=-1,2 \\
		1 & \text{if } p \text{ is odd.} \end{cases}$
	\end{enumerate}
\item If $a=p^\alpha u$ and $b=p^\beta v$ then 
	\[
	(a,b)_p = \begin{cases}
		(-1)^{\alpha\beta\epsilon(p)} \legendre{u}{p}^\beta \legendre{v}{p}^\alpha & \text{if } p\neq 2 \\
		(-1)^{\epsilon(u)\epsilon(v) + \alpha\omega(v) +\beta\omega(u)} & \text{if } p=2
		\end{cases}
	\]
	where $\epsilon(u)$ denotes the class modulo 2 of $\frac{u-1}{2}$ and $\omega(u)$ denotes the class modulo 2 of $\frac{u^2-1}{8}$ \cite[III.1.2, Theorem 1, page 20]{Se}.
\item By Dirichlet's Theorem, if $a$ and $m$ are relatively prime positive integers, there exists infinitely many primes $q$ such that $q\equiv a\mod m$ see \cite[III.2.2, Lemma 3 page 25]{Se}.
\item A quadratic form $f_p$ over $\Q_p$ is anisotropic if and only if its determinant $d(f)$ is in $(\Q_p^*)^2$ and $c_p(f_p)=-(-1,-1)_p$ \cite[IV.2.2, Theorem 6, page 36]{Se}.
\item By the Hasse principle, a quadratic form $f$ over $\Q$ is anisotropic if there is some prime $p$ for which the local form $f_p$ over $\Q_p$ is anisotropic \cite[IV.3, Theorem 8, page 41]{Se}.
\end{enumerate}
\end{lemma}

\begin{proof}[Proof of \autoref{allcusped}]
Let $W=\mathbb{H}^4/\Lambda$ so that $\Lambda$ is an arithmetic group commensurable to a group $\SO(f,\Z)$ for some non-degenerate quadratic form $f$ defined over $\Q$ of signature $(4,1)$. According to \cite[Theorems 6, 8]{Mon}, the commensurability class 
of $\SO(f,\Z)$ is uniquely determined by the projective equivalence class of $f$, which in turn, is itself determined by an invariant $S$ which is a product of $s$ many distinct odd primes, and another invariant $e_{-1}(F)$ (which we will not define here).
In our case because every $f$ has signature $(4,1)$, this invariant is always $2$ and so can be ignored. Hence the projective equivalence class of $f$ is completely determined by $S$.

Using \cite[Theorem 10]{Mon} (actually Claim 2 of the proof of Theorem 10), we may take $f$ to be the diagonal form (with basis $\{e_0,e_1,e_2,e_3,e_4\}$):
\begin{equation} \label{eq: MA forms}
 f = \begin{cases}
 \langle -1,1,1,aS,a \rangle & \text{if } S\equiv 1\mod 4 \\
 \langle 1,1,1,aS,-a\rangle & \text{if } S\equiv -1\mod 4
 \end{cases}
\end{equation} 
where $a$ is an odd prime such that $a\nmid S$, $a\equiv(-1)^s\mod4$ if $S\equiv 1\mod 4$ or $a\equiv(-1)^{s+1}\mod4$ if $S\equiv -1\mod 4$, and $\legendre{-a}{p}=-1$ for all $p\mid S$. The proof will be completed as a consequence of \cref{lem case 1,lem case 2} stated and proved below.
\end{proof}

\begin{lemma}\label{lem case 1}
Suppose $f = \langle -1,1,1,aS,a \rangle$ with $S\equiv 1\mod 4$ and $a\nmid S$, $a\equiv(-1)^s\mod4$ with $\legendre{-a}{p}=-1$ for all $p\mid S$ as in \cref{eq: MA forms}. Then $f$ contains infinitely many 
projectively inequivalent anisotropic quadratic subforms over $\Q$ of signature (3,1).
\end{lemma}

\begin{proof}
Suppose first that $s$ is even, so from the description of $a$ given above, $a\equiv 1\mod4$. By Lemma \ref{lem Se}(e) there exists infinitely many odd primes $q$ such that $q\equiv -S\mod8$.
%Let $v=(\frac{q-1}{2})e_0+(\frac{q+1}{2})e_1$.
Let $u=(\frac{q+1}{2})e_0+(\frac{q-1}{2})e_1$ so $f(u)=-q$. Then the diagonal form $f'=\langle -q,1,aS,a \rangle$ is a subform of $f$ with orthogonal basis $\{u,e_2,e_3,e_4\}$.

Over the local field $\Q_2$, the determinant $d(f')=-qa^2S$ is in $(\Q_2^*)^2$ since $q$ was chosen so that $-qS\equiv1\mod8$ (see \cref{lem Se}(b)). Using \cref{lem Se}(c), $c_2(f')$ simplifies to $(-q,S)_2(a,-S)_2$. However, since $S\equiv a\equiv 1\mod4$, it follows from \cref{lem Se}(d) that $c_2(f')=1=-(-1,-1)_2$. Therefore $f'$ is anisotropic over $\Q_2$ and thus also over $\Q$. Since different choices of $q$ yield projectively inequivalent forms, we get infinitely many $\Q$-inequivalent anisotropic quadratic subforms of signature (3,1).

Now suppose that $s$ is odd, so $a\equiv -1\mod4$. Since $s$ is odd, and $S\equiv 1\mod 4$ we can find a prime $p|S$ with $p\equiv 1\mod 4$. Let $u=(\frac{S+1}{2})e_0+(\frac{S-1}{2})e_1$ so $f(u)=-S$. Let $v=(\frac{S-1}{2})e_0+(\frac{S+1}{2})e_1 + m e_2$ where $m$ is any positive integer such that $p$ does not divide $m$. Then $f(v)=S+m^2=pS'+m^2$, and by \cref{lem Se}(a), $f(v)\in(\Q_p^*)^2$.
Therefore the diagonal form $f'=\langle -S,S+m^2,aS,a \rangle$ is a subform of $f$ with orthogonal basis $\{u,v,e_3,e_4\}$.
Over the local field $\Q_p$, the determinant $d(f')=-a^2S^2(S+m^2)$ is in $(\Q_p^*)^2$  since $\legendre{-1}{p}=1$ for $p\equiv 1\mod 4$ (see \cref{lem Se}(a)). As in the previous case, using \cref{lem Se}(c), $c_p(f')$ simplifies to
$$c_p(f') = (-S,aS)_p(-S,a)_p(aS,a)_p = (-S,a)_p(-S,a)_p(-S,a)_p = (-S,a)_p.$$
Now since $p|S$, $p\equiv 1\mod 4$ and $\legendre{-a}{p}=-1$, \cref{lem Se}(d) implies that $(-S,a)_p=-1=-(-1,-1)_p$. Therefore $f'$ is anisotropic over $\Q_p$ and thus also over $\Q$. Since different choices of $m$ yield projectively inequivalent forms, the conclusion follows as before.
\end{proof}

%\begin{rem} Note that \cref{lem case 1} absorbs \cref{immersedRT} with $S=1$ and $s$ even. In this case if we choose $a=1$ then we get the same argument.\end{rem}

\begin{lemma}\label{lem case 2}
Suppose $f = \langle 1,1,1,aS,-a \rangle$ with $S\equiv -1\mod 4$ and $a\nmid S$, $a\equiv(-1)^{s+1}\mod4$ with and $\legendre{-a}{p}=-1$ for all $p\mid S$ as in \cref{eq: MA forms}. Then $f$ contains infinitely many 
projectively inequivalent anisotropic quadratic subforms over $\Q$ of signature (3,1).
\end{lemma}

\begin{proof}
Suppose first that $s$ is even, so $a\equiv -1\mod4$. Pick $\alpha>\beta\geq0$ such that $aS(\alpha^2 S-\beta^2)\equiv-1\mod8$. Such $\alpha,\beta$ always exist.
To see this, if $a\equiv-1\mod8$ let $\alpha=1$ and $\beta=0$; if $a\equiv3\mod8$ and $S\equiv-1\mod8$ let $\alpha=2$ and $\beta=0$; and if $a\equiv3\mod8$ and $S\equiv3\mod8$ let $\alpha=3$ and $\beta=2$.
Note also that $aS(\alpha^2 S-\beta^2)>0$.

Set $u=\beta e_3+\alpha S e_4$ so that $f(u)=-aS(\alpha^2 S-\beta^2)<0$. Set $m=-f(u) >0$.
Then the diagonal form $f'=\langle 1,1,1,-m \rangle$ is a subform of $f$ with orthogonal basis $\{u,e_2,e_3,e_4\}$.
Since $m=-f(u)\equiv-1\mod 8$, it is not the sum of three squares and so $f'$ is anisotropic.
Since any other choice of $\alpha,\beta$ congruent to the particular $\alpha,\beta$ given as examples would still work, and different choices yield infinitely many projectively inequivalent forms, the conclusion follows as before.

Now suppose that $s$ is odd, so $a\equiv 1\mod4$. Fix a prime $p|S$ with $p\equiv -1\mod 4$. By Lemma \ref{lem Se}(e), there exists infinitely many primes $q\equiv 1\mod 4$ such that $\legendre{q}{p}=-1$.
Since $q\equiv 1\mod 4$ it can be written as a sum of two squares $q=\alpha^2+\beta^2$. Let $w_1=\alpha e_1+\beta e_2$ so $f(w_1)=q$.
Consider the diagonal quadratic form $g=\langle 1,q,-p^2\rangle$. \\[\baselineskip]
\noindent{\bf Claim 1:}~{\em $g$ represents $S$ over $\Q$.}\\[\baselineskip]
Assuming the claim for now, there exists an integer solution $x^2+qy^2-p^2z^2=Sm^2$. Let $w_2=xe_0+\beta y e_1-\alpha y e_2$ so $f(w_2)=Sm^2+p^2z^2$ and $w_2$ pairs trivially with $w_1$. Therefore, the diagonal form $f'=\langle q,Sm^2+p^2z^2,aS,-a\rangle$ is a subform of $f$ with orthogonal basis $\{w_1,w_2,e_3,e_4\}$.
Let $S'=S/p$. Then $f(w_2)=Sm^2+p^2z^2=p(S'm^2+pz^2)$, and since $S'm^2+pz^2\equiv S'm^2 \mod p$, $f(w_2)\equiv pS'\mod (\Q_p^*)^2$.
Considering $f'$ over $\Q_p$, since $d(f')=-qa^2m^2S(S+p^2z^2)$ and since $\legendre{q}{p}=-1$ implies $-q\in (\Q_p^*)^2$ (see \cref{lem Se}(a)), we see that $d(f')=-qa^2S(S+p^2z^2)\equiv p^2S'(S'+pz^2)\in (\Q_p^*)^2$.
Using \cref{lem Se}(c) and (d), $c_p(f')$ simplifies to 
\[
c_p(f')=-(S'+pz^2,p)_p \cdot (-qS',p)_p \cdot (q,p)_p \cdot(-a,p)_p=-1=-(-1,-1)_p.
\]
Therefore $f'$ is anisotropic over $\Q_p$ and thus also over $\Q$ (see \cref{lem Se}(g)). As before, different choices of $q$ yield projectively inequivalent forms, and the conclusion follows as before.

We now prove Claim 1. Since the determinant of $g$ is $-q$ up to squares, by the Hasse principle it suffices to show that $g$ represents $S$ over $k=\R,\Q_2,\Q_q$. By \cite[IV.2.2, Corollary to Theorem 6]{Se} the ternary quadratic form $g$ represents $S$ if $S\neq -q$ in $k^*/(k^*)^2$ or $(-1,q)=c_p(g)$. As $S>0$ and $-q<0$, the first case holds over $\R$. As $q\nmid S$, the first case also holds over $\Q_q$. Over $\Q_2$ we have $c_2(g)=(q,-p^2)_2=(q,-1)_2=(-1,q)_2$ as required. Therefore $g$ represents $S$ over $\Q$ which proves the claim.
\end{proof}

\section{Proof of \cref{main_or}}\label{proof}

 We begin with a general lemma.

\begin{lemma}
\label{positiveb1}
Let $X$ be an orientable finite volume hyperbolic $4$-manifold with $\chi(X)=1$ and containing an embedded orientable totally geodesic hyperbolic $3$-manifold. Then $b_1(X)>0$. \end{lemma}

\begin{proof} Let $N\hookrightarrow X$ be an embedded orientable totally geodesic hyperbolic $3$-manifold. Suppose that $N$ separates, then $X$ is decomposed into two finite volume hyperbolic 4-manifolds with geodesic boundary, whose
volumes are proportional to their (integral) Euler characteristic. However, $\chi(X)=1$, and this is a contradiction.  Duality now implies $b_1(X)>0$.\end{proof}
Note that the argument in the proof of \cref{positiveb1} also proves the following.
\begin{lemma}
\label{disjoint}
Let $X$ be an orientable finite volume hyperbolic $4$-manifold with $\chi(X)=1$ and which contains embedded orientable disjoint totally geodesic hyperbolic $3$-manifolds $N_1, N_2, \ldots , N_r$. Then $b_1(X)\geq r$.\end{lemma}

\begin{proof} The proof of \cref{positiveb1} shows that none of the $N_i$ can separate $X$ and furthermore, also shows that $N_1\cup N_2 \cup \ldots \cup N_r$ cannot separate $X$.  Thus $X\setminus N_1\cup N_2 \cup \ldots \cup N_r$ is connected, and standard argument now shows that $\pi_1(X)$ surjects a free group of rank $r$. This proves the lemma.\end{proof}

Henceforth, throughout this section $W$ is as in the statement of \cref{main_or} and $M\hookrightarrow W$ is a closed embedded orientable totally geodesic hyperbolic $3$-manifold.

Referring to \cref{RTmanifolds}, since the manifolds labelled $16-22$ have first Betti number equal $0$, we can apply \cref{positiveb1} to rule out these possibilities for $W$.

To deal with the remaining $15$ possibilities for $W$, observe that from \cref{RTmanifolds}, each of these manifolds admits at least one orientable cross-section.

\begin{lemma}
\label{isdisjoint1}
$M$ is disjoint from all orientable cross-sections.
\end{lemma}

\begin{proof}  By \cref{Xsection}, these cross-section are all isometric to one of the complements of the links $L$ stated in \cref{Xsection}.  Suppose that $M$ meets one of the cross-sections, then 
$M$ must meet $S^3\setminus L$ in a closed
orientable embedded totally geodesic surface. However, this is impossible by \cref{small}.\end{proof}
%
%\noindent We now use \cref{disjoint} as follows.
Since $M$ is disjoint from any orientable cross section, and from \cref{RTbackground} $W$ is a regular cover of ${\H}^4/\Gamma(2)$, using the isometries of $W$ induced from the reflections in the
co-ordinate hyperplanes we get at least $2$ disjoint copies of $M$ embedded in $W$ both of which are disjoint from the orientable cross-section, which is itself non-separating in $W$ (by the proof of \cref{positiveb1}).
Thus we can conclude from \cref{disjoint} that $b_1(W)\geq 3$.  Referring to \cref{RTmanifolds} we see that this excludes all examples except the first example listed in \cref{RTmanifolds}.  However, in this case, there are
$3$ orientable cross sections, and $M$ is disjoint from all of these by \cref{isdisjoint1}. This gives $6$ disjoint embedded copies of $M$, and so by \cref{disjoint} we actually have $b_1(W)>3$. This contradiction completes the proof. \qed

\begin{rem} We note here an alternative approach to ruling out the first $15$ manifolds of \cref{RTmanifolds}.  Using the three (resp. two) orientable cross-sections in the manifold $1$ (resp. manifolds $5$ and $6$) the six (resp. four) disjoint copies of $M$ embedded in $W$ together with \cref{disjoint} rules out these manifolds. For the manifolds numbered $7-15$ of \cref{RTmanifolds}  where $b_1=1$, 
the two disjoint copies of $M$ embedded in $W$, can be used together with \cref{disjoint} to rule these out.  This brings us to the manifolds $2$, $3$ and $4$.  In this case, from \cite[Table 2]{RT0}, the orientable cross-section is homeomorphic 
to the complement of $6_2^3$, in the case of manifold number $2$, and to the complement of $8_9^3$ for manifolds $3$ and $4$.  In these cases, that the orientable cross-sections are non-separating can be seen directly by checking 
that one of the cusp tori (say $T$) of the cross-section meets a $3$-torus cusp cross-section $C$ of the $4$-manifold. The torus $T$ is an embedded non-separating torus in $C$ and so we can find a dual curve in $C$ that meets $T$ once. 
It follows that the cross-sections have to be non-separating.
\end{rem}

\begin{rem} Using the equivalence of the quadratic forms $J$ and $J_7=x_0^2+x_1^2+x_2^2+7x_3^2-7x_4^2$ given in the proof of \cref{immersedRT},  one can construct explicit manifolds commensurable with any of the hyperbolic $4$-manifolds
considered in the proof of  \cref{main_or} containing a closed embedded orientable totally geodesic hyperbolic $3$-manifold.  

For example, if $\Gamma(49)  < \Or^+(J,\Z)$ denotes principal congruence subgroup of level $49$, then it can be checked that the equivalence described above conjugates $\Gamma(49)$ into a subgroup of the
principal congruence subgroup $\Gamma(7) < \Or^+(J_7,\Z)$.  The subform $x_0^2+x_1^2+x_2^2-7x_4^2$ defines a cocompact subgroup of $\Or^+(J_7,\Z)$ acting on a hyperbolic $3$-space $H$. Using the reflection in $\mathbb{H}^4$ through
$H$ and arguing as
in \cite{Mil}, it can be shown that $\mathbb{H}^4/\Gamma(7)$ contains a closed embedded orientable totally geodesic hyperbolic $3$-manifold, and hence so does the quotient of $\mathbb{H}^4$ by $\Gamma(49)  < \Or^+(J,\Z)$. The Euler characterstic
of $\mathbb{H}^4/\Gamma(49)$ is enormous, exceeding $700,000$.\end{rem}

\section{Volume from tubular neighbourhoods}
\label{tubular}
To prove \cref{main_1011}, we will make use of embedded totally geodesic hyperbolic 3-manifolds in a different way, and, in particular, we will make use of a result of Basmajian \cite{Bas} which provides disjoint collars about closed embedded orientable totally geodesic hypersurfaces in hyperbolic manifolds.  We state this only the case of interest, namely for hyperbolic $4$-manifolds.

Following \cite{Bas}, let $r(x) = \log \coth(x/2)$, let $V(r)$ denote the volume of a ball of radius $r$ in ${\H}^3$.  It is noted in \cite{Bas} that, $V(r) = \omega_3 \int_0^r \sinh^2(r)dr$, where $\omega_3$ is the area of the unit sphere in ${\R}^3$ (i.e. $\omega_3=4\pi$).

In \cite[pages 213--214]{Bas}, the volume of a tubular neighbourhood of a closed embedded orientable totally geodesic hyperbolic $3$-manifold of $3$-dimensional hyperbolic volume $A$ in a hyperbolic $4$-manifold is given
in terms of the the  $4$-dimensional tubular neighbourhood function $c_4(A) = (\frac{1}{2})(V \circ r)^{-1}(A)$. Moreover,
as noted in \cite[Remark 2.1]{Bas}, when the totally geodesic submanifold separates, an improved estimate can be obtained using the tubular neighbourhood function $d_4(A) = (\frac{1}{2})(V \circ r)^{-1}(A/2)$ and we record this as follows.

 \begin{lemma}
\label{volume_in_hood} 
Let $X$ be an orientable finite volume hyperbolic $4$-manifold containing a closed embedded separating orientable totally geodesic hyperbolic $3$-manifold of $3$-dimensional hyperbolic volume $A$. 
Then $X$ contains a tubular neighbourhood of $M$ of volume
$$\mathcal{V}'(A) = 2A\int_0^{d_4(A)}\cosh^3(t) dt.$$\end{lemma}

Moreover, \cite{Bas} also proves that disjoint embedded closed orientable totally geodesic hyperbolic $3$-manifolds in an orientable finite volume hyperbolic $4$-manifold have disjoint collars, thereby contributing additional volume.  For our purposes we summarize what we need in the following.

\begin{corollary}
\label{more_vol}
Let $X$ be an orientable finite volume hyperbolic $4$-manifold of Euler characteristic $\chi$ containing $K$ disjoint copies of a closed embedded orientable totally geodesic hyperbolic $3$-manifold of $3$-dimensional hyperbolic volume $A$. 
Assume that all of these disjoint copies separate $X$. Then
$$\Vol(X) = (\frac{4\pi^2}{3})\chi \geq K\mathcal{V}'(A).$$
\end{corollary}

\section{Proof of Theorem \ref{main_1011}}\label{proof1}

Let $W$ be as in the statement of \cref{main_1011}, and suppose that $M$ is a closed embedded totally geodesic hyperbolic $3$-manifold in $W$. Since $M\subset W\subset S^4$, $M$ is orientable and separates $W$. This will allow
us to use the formula for the volume of a tubular neighbourhood given in \cref{more_vol}.

Let $N$ be the non-orientable manifold $1011$ in the census of \cite{RT} with $W\rightarrow N$ the orientable double cover, and $L$ denote the link $8^3_9$. Note that again by the construction of $N$ in \cite{RT}, the manifold $N$ is a regular cover of $\mathbb{H}^4/\Gamma(2)$ with covering group $K$.

\begin{lemma}
\label{XsectionN}
In the case of the manifold $N$, each of the four cross sections is isometric to $S^3\setminus L$.\end{lemma}

\begin{proof} 
In \cite[Table 3]{RT} the manifold $1011$ is given by the code \texttt{14FF28} which represents the side pairing \texttt{11114444FFFFFFFF22228888} for the 24 sides of the ideal 24-cell $Q^4$. In the notation of \cite{RT}, the four cross sections have $k_1k_5k_9$ codes \texttt{714}, \texttt{274}, \texttt{172}, \texttt{147}, which correspond to the side pairings $r_ik_i$ for the 12 sides of the polytope $Q^3$ where $r_i$ is the reflection on side $i$ and $k_1=k_2=k_3=k_4$, $k_5=k_6=k_7=k_8$, $k_9=k_{10}=k_{11}=k_{12}$. Since $r_i$ is a reflection, the side pairing $r_ik_i$ is orientation preserving if and only if the corresponding $k_i$ is orientation reversing. But this happens only if $k_i\in\{1,2,4,7\}$ since then it corresponds to the diagonal matrices with $1\leftrightarrow\mathrm{diag}(-1,1,1,1)$, $2\leftrightarrow\mathrm{diag}(1,-1,1,1)$, $4\leftrightarrow\mathrm{diag}(1,1,-1,1)$, $7\leftrightarrow\mathrm{diag}(-1,-1,-1,1)$. Therefore, all four cross sections of $N$ are 
 
 orientable.

In \cite[Table 1]{RT} we see that the code \texttt{147} corresponds to the integral congruence two $3$-manifold $M_2^3$ of \cite{RT} and which is isometric to the link complement $S^3\setminus L$ \cite[Page 108]{RT}. The other three codes \texttt{714}, \texttt{274}, \texttt{172} do not appear in \cite[Table 1]{RT}. However, as we briefly describe below, these are equivalent up to symmetries of $Q^3$ to \texttt{147}.

First, the symmetries of $Q^3$ are identified with the symmetries of the cube whose vertices are the actual vertices of $Q^3$ (see \cref{fig: Q3}). Using this identification, the codes \texttt{714} and \texttt{274} are equivalent to \texttt{147} via a rotation by $\pi$ along axes between the midpoints of opposite edges of the cube, and the code \texttt{172} is equivalent to \texttt{147} via a rotation by $\pi$ along an axis between the centers of two opposite faces.\end{proof}
%{\red MC: For this claim, I wrote the Mathematica file called "RTCodeEquiv.nb" to compute the 3-d codes up to symmetries of $Q^3$.}

\begin{figure}
 \includegraphics[width=.5\linewidth]{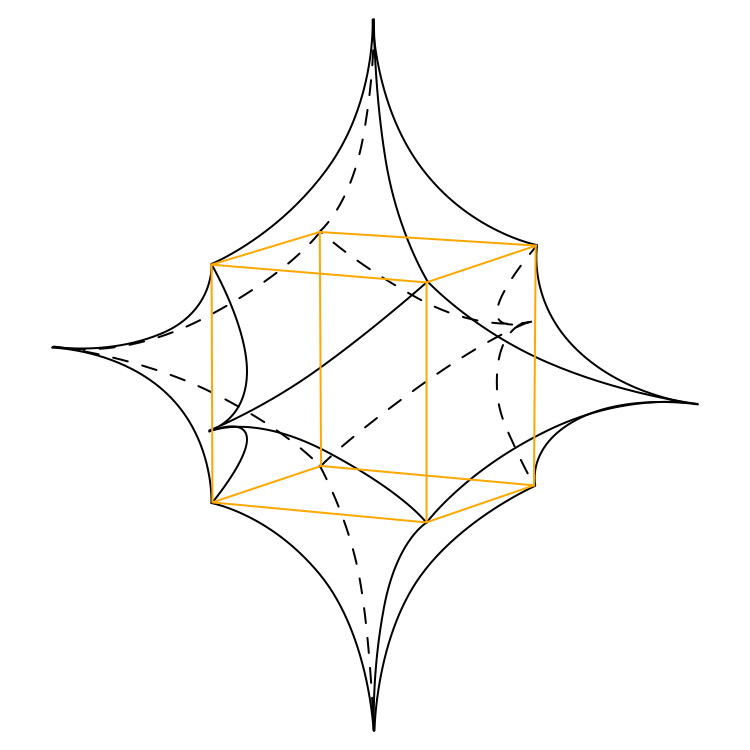}
 \caption{The polytope $Q^3$}
 \label{fig: Q3}
\end{figure}

\begin{lemma}
\label{isdisjoint}
\begin{enumerate}
\item $W$ is a regular cover of ${\H}^4/\Gamma(2)$.
\item The lift of any cross section of $N$ to $W$ consists of two embedded totally geodesic copies of $S^3\setminus L$.
\item $M$ is disjoint from all such lifts.
\end{enumerate}
\end{lemma}

\begin{proof} For (1) we note that $W$ is the orientable double cover of $N$, as such it is a characteristic cover of $N$. Now $N$ is a regular cover of ${\H}^4/\Gamma(2)$, hence $W$ is a regular cover of ${\H}^4/\Gamma(2)$. 

For (2), we have from \cref{XsectionN} that all cross sections of $N$ are isometric to the link complement $S^3\setminus L$. Being orientable, these must lift to two copies in the orientable double cover.

For (3),  we argue as in the proof of \cref{isdisjoint1}.\end{proof}
\noindent The first part of \cref{main_1011} now follows from  \cref{isdisjoint}(2).\\[\baselineskip]
\noindent  For the second part, note that from \cref{isdisjoint}(3) and (1), since $M$ is disjoint from all lifts of the cross sections, and $W$ is a regular cover of ${\H}^4/\Gamma(2)$, using the isometries of $W$ induced from the reflections in the
co-ordinate hyperplanes we get $16$ disjoint copies of $M$, all embedded and separating in $W$. Now the minimal volume of a closed hyperbolic 3-manifold is that of the Weeks manifold and is approximately $0.9427\ldots$ \cite{GMM}.
Using this estimate for $\Vol(M)$,  and applying
\cref{more_vol} we see that $\Vol(W) \geq 16\mathcal{V}'(0.94)$, which  is approximately $28.9$. On the other hand, since $\chi(W)=2$, $\Vol(W)= \frac{8\pi^2}{3}$ which is approximately $26.3$, a contradiction.\qed

%\begin{rem} In \cite{IRT}, three other examples of link complements of $2$-tori were found in manifolds homeomorphic to $S^4$. These arise as the orientable double covers of the non-orientable manifolds
%in the census of \cite{RT} with numbers $23$, $71$ and $1091$.   Although we have not carried this out, a similar strategy to that used in the proof of \cref{main_1092} could likely be employed to attempt to deal with these other cases, with the additional complication that  the manifolds corresponding to  $71$ and $1091$ only have one orientable cross-section, and $23$ has none.\end{rem}
%Using the methods of proof of Theorem \ref{main_or} and \ref{main_1011},  it can shown that the link complement of $2$-tori covering $1092$ also does not contain any closed embedded totally geodesic $3$-manifolds.

\begin{rem}\label{rem:Sar} In \cite{Sar} an investigation of finite volume hyperbolic link complements of $2$-tori and Klein bottles in other smooth, closed, simply connected $4$-manifolds was initiated. Amongst other things, this work provided restrictions on the simply connected manifolds that can admit such link complements; namely they can only be homeomorphic to $S^4$, $\#_{r}(S^2\times S^2)$, or $\#_r (\mathbb{CP}^2 \# \overline{\mathbb{CP}^2})$, with $r>0$.  Furthermore, using the examples of \cite{I}, examples of link complements
of $2$-tori in $\#_{r}(S^2\times S^2)$ for $r$ even were exhibited in \cite{Sar} (these cover the manifold $W$ above).  Other examples of link complements of $2$-tori and Klein bottles in closed simply connected manifolds are also given in \cite{IRT}.   

Note that for the example $W$ of \cite{I} considered in \cref{main_1011}, it is shown in \cite{I2} that the link complement is in $S^4$ with the standard smooth
structure.  \end{rem}

\begin{rem} It is known that every closed orientable $3$-manifold embeds in $\#_{r}(S^2\times S^2)$, for some $r>0$ (see \cite[Chapter VII, Theorem 4]{Kir0}). On the other hand, we do not know whether the link complements 
of $2$-tori in $\#_{r}(S^2\times S^2)$ that cover $W$ mentioned in \cref{rem:Sar} contain a closed embedded totally geodesic hyperbolic 3-manifold.\end{rem}

Motivated by the results of this paper, these remarks and recent work on embedding (arithmetic) hyperbolic manifolds as co-dimension one totally geodesic submanifolds (see \cite{KRS} and references therein) we pose the following questions:

\begin{quest}
\label{questone}
Is there a cusped, orientable, finite volume hyperbolic $4$-manifold $W$ with $\chi(W)=1$ (or $2$) which contains a closed embedded orientable totally geodesic hyperbolic $3$-manifold? If not what is the minimal Euler characteristic of such a hyperbolic $4$-manifold?\end{quest}

\begin{quest}
\label{questtwo}
Do any of the link complements of $2$-tori in $\#_{r}(S^2\times S^2)$ that cover $W$ mentioned in Remark \ref{rem:Sar} contain a closed embedded orientable totally geodesic hyperbolic 3-manifold?\end{quest}

\begin{quest}
\label{questhree}
Does there exist a finite volume hyperbolic link complement of $2$-tori and Klein bottles in $\#_r (\mathbb{CP}^2 \# \overline{\mathbb{CP}^2})$, for some $r>0$?\end{quest}

\vfill\eject

\section{The orientable integral congruence two hyperbolic 4-manifolds}
\label{RTmanifolds}
The following table is composed from data in  \cite[Table 2]{RT0} and \cite[Table 2]{RT}.

\begin{table}[h]
\caption{The $22$ orientable integral congruence two hyperbolic $4$-manifolds of \cite{RT}}
%\label{table:MoreLinkComplements}}
\begin{center}
\begin{tabular}{|c|c|c|}
\hline
Number&$b_1$&Number of orientable cross-sections\\
\hline
$1$&$3$&$3$\\
\hline
$2$&$2$&$1$\\
\hline
$3$&$2$&$1$\\
\hline
$4$&$2$&$1$\\
\hline
$5$&$2$&$2$\\
\hline
$6$&$2$&$2$\\
\hline
$7$&$1$&$1$\\
\hline
$8$&$1$&$1$\\
\hline
$9$&$1$&$1$\\
\hline
$10$&$1$&$1$\\
\hline
$11$&$1$&$1$\\
\hline
$12$&$1$&$1$\\
\hline
$13$&$1$&$1$\\
\hline
$14$&$1$&$1$\\
\hline
$15$&$1$&$1$\\
\hline
$16$&$0$&$0$\\
\hline
$17$&$0$&$0$\\
\hline
$18$&$0$&$0$\\
\hline
$19$&$0$&$0$\\
\hline
$20$&$0$&$0$\\
\hline
$21$&$0$&$0$\\
\hline
$22$&$0$&$0$\\
\hline
\end{tabular}
\end{center}
%%\centerline{Table 1}
\end{table}

%\bibliography{refs}

\begin{thebibliography}{CGLS}

\bibitem[ALR]{ALR} I. Agol, D. D. Long and A. W. Reid, {\em The Bianchi groups are separable on geometrically finite subgroups}, Annals of Math. {\bf 153} (2001), 599--621.

\bibitem[Bas]{Bas} A. Basmajian, {\em Tubular neighborhoods of totally geodesic hypersurfaces in hyperbolic manifolds}, Invent. Math. {\bf 117} (1994), 207--225.

\bibitem[BHCh]{BHCh} A. Borel and Harish-Chandra, {\em Arithmetic subgroups of algebraic groups}, Ann. Math. {\bf 75} (1962), 485--535.

%\bibitem[Mag]{Mag}  W.~Bosma, J.~Cannon, C.~Playoust, \emph{The Magma algebra system. I. The user language}, J. Symbolic Comput.,  {\bf 24} (1997), 235--265.

\bibitem[BB]{BB} R. Budney and B. A. Burton,  {\em Embeddings of 3-manifolds in $S^4$ from the point of view of the $11$-tetrahedron census}, arXiv:0810.2346.

\bibitem[CM]{CM} C. Cao and G. R. Meyerhoff, {\em The orientable cusped hyperbolic 3-manifolds of minimum volume}, Invent. Math. {\bf 146} (2001), 451--478. 

\bibitem[CR]{CR} M. Chu and A. W. Reid, {\em Totally geodesic hyperbolic 3-manifolds in hyperbolic link complements of tori in $S^4$}, in preparation.

\bibitem[CGLS]{CGLS} M. Culler, C. McA. Gordon, J. Luecke and P. B. Shalen, {\em Dehn surgery on knots}, Annals of Math. {\bf 125} (1987), 237--300.

\bibitem[CJR]{CJR} M. Culler, W. Jaco, and H. Rubinstein, {\em Incompressible surfaces in once-punctured torus bundles}, Proc. London Math. Soc. {\bf 45} (1982), 385--419.

\bibitem[FH]{FH} W. Floyd and A. Hatcher, {\em Incompressible surfaces in punctured-torus bundles}, Topology Appl. {\bf 13} (1982), 263--282.

\bibitem[GMM]{GMM} D. Gabai, R. Meyerhoff and P. Milley, {\em Minimum volume cusped hyperbolic three-manifolds}, J. Amer. Math. Soc. {\bf 22} (2009), 1157--1215.

\bibitem[Han]{Han} W. Hantzsche, {\em Einlagerung von Mannigflitigkeiten in euklidishe Raume}, Math. Zeit. {\bf 43} (1938), 38--58.

\bibitem[HT]{HT} A. Hatcher and W. Thurston, {\em Incompressible surfaces in $2$-bridge knot complements}, Invent. Math. {\bf 79} (1985), 225--246.

\bibitem[I0]{I0} D. Ivansic, {\em Embeddability of noncompact hyperbolic manifolds as complements of codimension-1 and -2 submanifolds}, In memory of T. Benny Rushing. Topology Appl. {\bf 120} (2002), 211--236. 

\bibitem[I]{I} D. Ivansic, {\em Hyperbolic structure on a complement of tori in the $4$-sphere}, Adv. Geom. {\bf 4} (2004), 119--139.

\bibitem[I2]{I2}  D. Ivansic, {\em A topological 4-sphere that is standard}, Adv. Geom. {\bf 12} (2012), 461--482.

\bibitem[IRT]{IRT} D. Ivansic, J. G. Ratcliffe and S. T. Tschantz, {\em Complements of tori and Klein bottles in the $4$-sphere that have hyperbolic structure}, Algebraic and Geometric Topology, {\bf 5} (2005), 999--1026.

%\bibitem[JR]{JR} J. Jung and A. W. Reid, {\em Embedding closed totally geodesic surfaces in Bianchi orbifolds},  arXiv:2003.05427.

\bibitem[Ki0]{Kir0} R. C. Kirby, {\em The Topology of $4$-Manifolds}, L.N.M. {\bf 1374}, Springer-Verlag (1989).

\bibitem[Kir]{Kir} R. Kirby, {\em Problems in low-dimensional topology}, Edited by R. Kirby. AMS/IP Stud. Adv. Math., 2.2, Geometric topology (Athens, GA, $1993$), 35--473, Amer. Math. Soc., Providence, RI, (1997).

%\bibitem[KM]{KM} S. Kojima and Y. Miyamoto, {\em The smallest hyperbolic 3-manifolds with totally geodesic boundary}, J. Diff. Geom. {\bf 34} (1991), 175--192. %%AR: maybe not needed now

\bibitem[KRS]{KRS} A. Kolpakov, A. W. Reid, and L. Slavich, {\em Embedding arithmetic hyperbolic manifolds}, Math. Res. Lett. {\bf 25} (2018), 1305--1328.

\bibitem[Lei]{Lei} C. J. Leininger, {\em Small curvature surfaces in hyperbolic 3-manifolds}, J. Knot Theory Ramifications {\bf 15} (2006), 379--411.

\bibitem[Loz]{Loz} M-T. Lozano, {\em Arcbodies}, Math. Proc. Camb. Phil. Soc. {\bf 94} (1983), 253--260.

%\bibitem[MR0]{MR0} C. Maclachlan and A. W. Reid, {\em Parameterizing Fuchsian subgroups of Bianchi groups}, Canadian J. Math. {\bf 43} (1991), 158--181.

\bibitem[MR]{MR} C. Maclachlan and A. W. Reid, {\em The Arithmetic of Hyperbolic $3$-Manifolds}, Graduate Texts in Mathematics, {\bf 219}, Springer-Verlag (2003).

\bibitem[MeR]{MeR} W. Menasco and A. W. Reid, {\em Totally geodesic surfaces in hyperbolic link complements}, in Topology $'90$ (Columbus, OH, $1990$), 215--226, Ohio State Univ. Math. Res. Inst. Publ. {\bf 1}, de Gruyter, Berlin, (1992).

\bibitem[Mil]{Mil} J. J. Millson, {\em On the first Betti number of a constant negatively curved manifold}, Annals of Math. {\bf 104} (1976), 235--247.

\bibitem[Mon]{Mon} J. M. Montesinos-Amilibia, {\em On odd rank integral quadratic forms: canonical representatives of projective classes and explicit construction of integral classes with square-free determinant}, Rev. R. Acad. Cienc. Exactas Fís. Nat. Ser. A Mat. RACSAM {\bf 109} (2015), 199--245.

\bibitem[Oe]{Oe} U. Oertel, {\em Closed incompressible surfaces in complements of star links}, Pacific J. Math. {\bf 111} (1984), 209--230.

\bibitem[RT0]{RT0} J. G. Ratcliffe and S. T. Tschantz, {\em Gravitational instantons of constant curvature}, in Topology of the Universe Conference (Cleveland, OH, 1997), Classical Quantum Gravity {\bf 15} (1998), 2613--2627.

\bibitem[RT]{RT} J. G. Ratcliffe and S. T. Tschantz, {\em The volume spectrum of hyperbolic $3$-manifolds}, Experimental Math. {\bf 9} (2000), 101--125.

\bibitem[RiSl]{RiSl} S. Riolo and L. Slavich, {\em New hyperbolic $4$-manifolds of low volume}, Algebraic and Geometric Topology, {\bf 19} (2019), 2653--2676.

\bibitem[Sar]{Sar} H. Saratchandran, {\em Finite volume hyperbolic complements of $2$-tori and Klein bottles in closed smooth simply connected $4$-manifolds}, New York J. Math. {\bf 24} (2018), 443--450. 

\bibitem[Se]{Se}  J-P. Serre, {\em A Course in Arithmetic}, Graduate Texts in Math. {\bf 7} Springer-Verlag (1973).

%\bibitem[Sw]{Sw}  R.G. Swan, {\em Generators and relations for certain special linear groups}, Advances in Math. {\bf 6} (1971), 1--77.

\bibitem[Th]{Th} W. P. Thurston, {\em The Geometry and Topology of 3-Manifolds}, Princeton University mimeographed notes, (1979).

\bibitem[VS]{VS} E. B. Vinberg and O. V. Shvartsman, {\em Discrete groups of motions of spaces of constant curvature}, in Geometry II, Encyc. Math. Sci. {\bf 29} Springer (1993), 139--248.

\bibitem[Wu]{Wu} Y-Q. Wu, {\em Incompressibility of surfaces in surgered $3$-manifolds}, Topology {\bf 31} (1992), 271--279.

\end{thebibliography}
%\bibliographystyle{alpha}

%\bibitem[BCP97]{Mag} W. Bosma, J. Cannon, and C. Playoust, {\em The Magma algebra system. I. The user language}, J. Symbolic Comput.,  {\bf 24} (1997), 235--265.

%\bibitem[CDGW]{CDGW} M. Culler, N. M. Dunfield, M. Goerner, and J. R. Weeks, {\em SnapPy, a computer program for studying the geometry and topology of 3-manifolds}, \url{http://snappy.computop.org}, version 2.6 (12/29/2017).

\end{document}